\documentclass{amsart}
\usepackage{amsfonts}
\usepackage{hyperref}

\newtheorem{theorem}{Theorem}
\theoremstyle{plain}

\newtheorem{corollary}{Corollary}

\newtheorem{definition}{Definition}

\newtheorem{lemma}{Lemma}

\newtheorem{proposition}{Proposition}
\newtheorem{remark}{Remark}

\numberwithin{equation}{section}
\input{tcilatex}

\begin{document}
\title[On Lazarevi\'{c} and Cusa type inequalities for hyperbolic functions]{%
On Lazarevi\'{c} and Cusa type inequalities for hyperbolic functions with
two parameters}
\author{Zhen-Hang Yang}
\address{Power Supply Service Center, Zhejiang Electric Power Corporation
Research Institute, Hangzhou City, Zhejiang Province, 310009, China}
\email{yzhkm@163.com}
\date{March 31, 2014}
\thanks{This paper is in final form and no version of it will be submitted
for publication elsewhere.}

\begin{abstract}
In this paper, by investigating the monotonicity of a function composed of $%
\left( \sinh x\right) /x$ and $\cosh x$ with two parameters in $x$ on $%
\left( 0,\infty \right) $, we prove serval theorems related to inequalities
for hyperbolic functions, which generalize known results and establish some
new and sharp inequalities. As applications, some new and sharp inequalities
for bivariate means are presented.
\end{abstract}

\maketitle

\section{Introduction}

Lazarevi\'{c} \cite{Lazarevic.SMF.170.1966} (or see Mitrinovi\'{c} \cite%
{Mitrinovic.AI.Springer.1970}) proved that for $x\neq 0$, the inequality%
\begin{equation}
\left( \frac{\sinh x}{x}\right) ^{q}>\cosh x  \label{I.L}
\end{equation}%
holds if and only if $q\geq 3$. This result has been generalized by Zhu in 
\cite{Zhu.2009.JIA.379142} as follows.

\noindent \textbf{Theorem Zhu1. }\emph{Let }$p>1$\emph{\ or }$p\leq 8/15$%
\emph{, and }$x\in \left( 0,\infty \right) $\emph{. Then}%
\begin{equation*}
\left( \frac{\sinh x}{x}\right) ^{q}>p+\left( 1-p\right) \cosh x
\end{equation*}%
\emph{if and only if }$q\geq 3\left( 1-p\right) $\emph{.}

Yang gave another generalization and refinement in \cite{Yang.JIA.2013.116}.

\noindent \textbf{Theorem Yang}. \emph{Let }$p,x>0$\emph{. Then the
following inequality }%
\begin{equation}
\frac{\sinh x}{x}>\left( \cosh px\right) ^{1/\left( 3p^{2}\right) }
\label{L-A_pG}
\end{equation}%
\emph{holds for all }$x>0$\emph{\ if and only if}$\ p\geq p_{0}=1/\sqrt{5}$%
\emph{,\ and the function }$p\mapsto \left( \cosh px\right) ^{1/\left(
3p^{2}\right) }$\emph{\ is decreasing on }$\left( 0,\infty \right) $\emph{.
Inequality (\ref{L-A_pG}) is reversed if and only if }$0<p\leq 1/3$\emph{. }

Another inequality related to Lazarevi\'{c} inequality is the so-called Cusa
type one (see \cite{Neuman.MIA.13.4.2010}), which states that%
\begin{equation}
\frac{\sinh x}{x}<\frac{2+\cosh x}{3}  \label{Cusah}
\end{equation}%
holds for $x>0$.

In \cite{Zhu.JIA.2010.130821}, Zhu established a more general result which
contains Lazarevi\'{c} and Cusa-type inequalities.

\noindent \textbf{Theorem Zhu2.} \emph{Let }$x>0$\emph{. Then the following
are considered.}

\emph{(i) If }$p\geq 4/5$\emph{, the double inequality}%
\begin{equation*}
1-\lambda +\lambda \left( \cosh x\right) ^{p}<\left( \frac{\sinh x}{x}%
\right) ^{p}<1-\eta +\eta \left( \cosh x\right) ^{p}
\end{equation*}%
\emph{holds if and only if }$\eta \geq 1/3$\emph{\ and }$\lambda \leq 0$%
\emph{.}

\emph{(ii) If }$p<0$\emph{, the inequality}%
\begin{equation*}
\left( \frac{\sinh x}{x}\right) ^{p}<1-\eta +\eta \left( \cosh x\right) ^{p}
\end{equation*}%
\emph{holds if and only if }$\eta \leq 1/3$\emph{.}

\emph{That is, let }$\alpha >0$\emph{, then the inequality}%
\begin{equation*}
\left( \frac{x}{\sinh x}\right) ^{\alpha }<1-\eta +\eta \left( \frac{1}{%
\cosh x}\right) ^{\alpha }
\end{equation*}%
\emph{holds if and only if }$\eta \leq 1/3$\emph{.}

Other inequalities for hyperbolic functions can be found in \cite%
{Zhu.MIA.11.3.2008}, \cite{Zhu.CMA.58.2009.a}, \cite{Zhu.AAA.2009}, \cite%
{Neuman.MIA.13.4.2010}, \cite{Sandor.arXiv.1105.0859v1.2011}, \cite%
{Yang.JMI.6.4.2012}, \cite{Neuman.AIA.1.1.2012}, \cite{Wu.AML.25.5.2012} 
\cite{Neuman.AMC.218.2012}, \cite{Sandor.MIA.15.2.2012}, \cite%
{Yang.JIA.2013.116}, \cite{Zhu.JIA.2012.303}, \cite{Chen.JMI.8.1.2014}, \cite%
{Yang.AAA.2014.364076}, and references therein.

The aim of this paper is to establish more general than Zhu's inequalities
for hyperbolic functions. In Section 2, we investigate the monotonicity of
the function $H_{p,q}$ defined on $\left( 0,\infty \right) $ by%
\begin{equation}
H_{p,q}\left( x\right) =\frac{U_{p}\left( \frac{\sinh x}{x}\right) }{%
U_{q}\left( \cosh x\right) },  \label{H}
\end{equation}%
where $p,q\in \mathbb{R}$ and $U_{p}$ is defined on $\left( 1,\infty \right) 
$ by%
\begin{equation}
U_{p}\left( t\right) =\frac{t^{p}-1}{p}\text{ if }p\neq 0\text{ and }%
U_{0}\left( t\right) =\ln t.  \label{U_p}
\end{equation}%
If we can prove that $H_{p,q}$ is increasing or decreasing on $\left(
0,\infty \right) $ for certain $p,q$, then we will obtain $H_{p,q}\left(
x\right) >\left( \text{or}<\right) H_{p,q}\left( 0^{+}\right) =1/3$, which
may yield some new inequalities for hyperbolic functions. Our main purpose
in the section is to find the relations between $p$ with $q$ such that $%
H_{p,q}$ has monotonicity property. Based on them, many new sharp
inequalities for hyperbolic functions are derived in Section 3. In the last
section, some new sharp inequalities for bivariate means are presented.

\section{Monotonicity}

We begin with the following simple assertion.

\begin{lemma}
\label{Lemma u_p}Let the function $U_{p}$ defined on $\left( 1,\infty
\right) $ by (\ref{U_p}). Then $p\mapsto U_{p}\left( t\right) $ is
increasing on $\mathbb{R}$ and $U_{p}\left( t\right) >0$ for $t\in \left(
1,\infty \right) $.
\end{lemma}

\begin{proof}
For $p\neq 0$, differentiation yields%
\begin{equation*}
\frac{\partial U_{p}\left( t\right) }{\partial p}=-\frac{1}{p^{2}}\left(
t^{p}-1\right) +\frac{1}{p}t^{p}\ln t=-\frac{t^{p}}{p^{2}}\left( \ln
t^{-p}-\left( t^{-p}-1\right) \right) >0,
\end{equation*}%
where the last inequality holds due to $\ln x\leq \left( x-1\right) $ for $%
x>0$.

Employing the decreasing property, we get%
\begin{equation*}
U_{p}\left( t\right) >\lim_{p\rightarrow -\infty }U_{p}\left( t\right)
=\lim_{p\rightarrow -\infty }\frac{t^{p}-1}{p}=0,
\end{equation*}%
which proves the lemma.
\end{proof}

For $x\in (0,\infty )$, we denote by%
\begin{equation*}
Sh_{p}\left( x\right) :=U_{p}\left( \tfrac{\sinh x}{x}\right) \text{ \ and \ 
}Ch_{p}\left( x\right) :=U_{p}\left( \cosh x\right)
\end{equation*}%
due to $\left( \sinh x\right) /x,\cosh x\in \left( 1,\infty \right) $. Then
we have%
\begin{eqnarray}
Sh_{p}\left( x\right) &=&\frac{\left( \frac{\sinh x}{x}\right) ^{p}-1}{p}%
\text{ if }p\neq 0\text{ \ and \ }Sh_{0}\left( x\right) =\ln \frac{\sinh x}{x%
}\text{ if }p=0,  \label{Sh_p} \\
Ch_{p}\left( x\right) &=&\frac{\cosh ^{p}x-1}{p}\text{ if }p\neq 0\text{ \
and \ }Ch_{0}\left( x\right) =\ln \left( \cosh x\right) \text{ if }p=0.
\label{Ch_p}
\end{eqnarray}%
And then, the function $x\mapsto H_{p,q}\left( x\right) =U_{p}\left( \frac{%
\sinh x}{x}\right) /U_{q}\left( \cosh x\right) =Sh_{p}\left( x\right)
/Ch_{q}\left( x\right) $ can be expressed as%
\begin{equation}
H_{p,q}\left( x\right) =\left\{ 
\begin{array}{ll}
\frac{q}{p}\frac{\left( \frac{\sinh x}{x}\right) ^{p}-1}{\cosh ^{q}x-1} & 
\text{if }pq\neq 0,\bigskip \\ 
\frac{1}{p}\frac{\left( \frac{\sinh x}{x}\right) ^{p}-1}{\ln \left( \cosh
x\right) } & \text{if }p\neq 0,q=0,\bigskip \\ 
q\frac{\ln \frac{\sinh x}{x}}{\cosh ^{q}x-1} & \text{if }p=0,q\neq 0,\bigskip
\\ 
\frac{\ln \frac{\sinh x}{x}}{\ln \left( \cosh x\right) } & \text{if }%
p=q=0.\bigskip%
\end{array}%
\right.  \label{H_p,q}
\end{equation}%
In order to investigate the monotonicity of the function $H_{p,q}$, we first
recall the following lemmas.

\begin{lemma}[\protect\cite{Vamanamurthy.183.1994}, \protect\cite%
{Anderson.New York. 1997}]
\label{Lemma monotonicity of ratio}Let $f,g:\left[ a,b\right] \mapsto 
\mathbb{R}$ be two continuous functions which are differentiable on $\left(
a,b\right) $. Further, let $g^{\prime }\neq 0$ on $\left( a,b\right) $. If $%
f^{\prime }/g^{\prime }$ is increasing (or decreasing) on $\left( a,b\right) 
$, then so are the functions 
\begin{equation*}
x\mapsto \frac{f\left( x\right) -f\left( a\right) }{g\left( x\right)
-g\left( a\right) }\text{ \ \ \ and \ \ \ }x\mapsto \frac{f\left( x\right)
-f\left( b\right) }{g\left( x\right) -g\left( b\right) }.
\end{equation*}
\end{lemma}

\begin{lemma}[\protect\cite{Biernacki.9.1955}]
\label{Lemma monotocity of A/B}Let $a_{n}$ and $b_{n}$ $(n=0,1,2,...)$ be
real numbers and let the power series $A\left( t\right) =\sum_{n=0}^{\infty
}a_{n}t^{n}$ and $B\left( t\right) =\sum_{n=0}^{\infty }b_{n}t^{n}$ be
convergent for $|t|<R$. If $b_{n}>0$ for $n=0,1,2,...$, and $a_{n}/b_{n}$ is
strictly increasing (or decreasing) for $n=0,1,2,...$, then the function $%
A\left( t\right) /B\left( t\right) $ is strictly increasing (or decreasing)
on $\left( 0,R\right) $.
\end{lemma}

Now we are in position to prove the monotonicity of $H_{p,q}$. Clearly, $%
H_{p,q}\left( x\right) $ can be written as%
\begin{equation*}
H_{p,q}\left( x\right) =\frac{Sh_{p}\left( x\right) }{Ch_{q}\left( x\right) }%
=\frac{Sh_{p}\left( x\right) -Sh_{p}\left( 0^{+}\right) }{Ch_{q}\left(
x\right) -Ch_{q}\left( 0^{+}\right) }.
\end{equation*}%
For $pq\neq 0$, differentiation yields%
\begin{equation}
\frac{Sh_{p}^{\prime }\left( x\right) }{Ch_{q}^{\prime }\left( x\right) }=%
\frac{\cosh ^{1-q}x}{x^{2}\sinh x}\left( \frac{\sinh x}{x}\right)
^{p-1}\left( x\cosh x-\sinh x\right) :=f_{1}(x)  \label{f1}
\end{equation}%
\begin{equation}
f_{1}^{\prime }(x)=\frac{1}{x^{2}\sinh ^{3}x\cosh ^{q}x}\left( \frac{\sin x}{%
x}\right) ^{p}\times f_{2}\left( x\right) ,  \label{df1}
\end{equation}%
where%
\begin{equation}
f_{2}\left( x\right) =pA\left( x\right) -qB\left( x\right) +C\left( x\right)
,  \label{f2}
\end{equation}%
in which 
\begin{subequations}
\begin{align}
A\left( x\right) & =\left( \sinh x-x\cosh x\right) ^{2}\cosh x>0,  \label{A}
\\
B\left( x\right) & =x\left( x\cosh x-\sinh x\right) \sinh ^{2}x>0  \label{B}
\\
C\left( x\right) & =-2x^{2}\cosh x+x\sinh x+\cosh x\sinh ^{2}x>0,  \label{C}
\end{align}%
here $C\left( x\right) >0$ due to 
\end{subequations}
\begin{equation*}
C\left( x\right) =x^{2}\left( \cosh x\right) \left( \frac{\sinh ^{2}x}{x^{2}}%
+\frac{\tanh x}{x}-2\right) >0
\end{equation*}%
by Wilker type inequality (see \cite{Zhu.MIA.10.4.2007}). It is easy to
verify that (\ref{f1}), (\ref{df1}) and (\ref{f2}) are true for $pq=0$.

Expanding in power series yields%
\begin{eqnarray}
A\left( x\right) &=&\frac{1}{4}x^{2}\cosh 3x+\frac{3}{4}x^{2}\cosh x-\frac{1%
}{2}x\sinh 3x-\frac{1}{2}x\sinh x+\frac{1}{4}\cosh 3x-\frac{1}{4}\cosh x 
\notag \\
&=&\frac{1}{4}\sum_{n=0}^{\infty }\frac{3^{2n}}{\left( 2n\right) !}x^{2n+2}+%
\frac{3}{4}\sum_{n=0}^{\infty }\frac{1}{\left( 2n\right) !}x^{2n+2}-\frac{1}{%
2}\sum_{n=1}^{\infty }\frac{3^{2n-1}}{\left( 2n-1\right) !}x^{2n}  \notag \\
&&-\frac{1}{2}\sum_{n=1}^{\infty }\frac{1}{\left( 2n-1\right) !}x^{2n}+\frac{%
1}{4}\sum_{n=0}^{\infty }\frac{3^{2n}}{\left( 2n\right) !}x^{2n}-\frac{1}{4}%
\sum_{n=0}^{\infty }\frac{1}{\left( 2n\right) !}x^{2n}  \notag \\
&=&\sum_{n=3}^{\infty }\frac{\left( \left( 4n^{2}-14n+9\right)
9^{n-1}+12n^{2}-10n-1\right) }{4\left( 2n\right) !}x^{2n}:=\sum_{n=3}^{%
\infty }\frac{a_{n}}{4\left( 2n\right) !}x^{2n},  \label{Ae}
\end{eqnarray}%
\begin{eqnarray}
B\left( x\right) &=&\frac{1}{4}x^{2}\cosh 3x-\frac{1}{4}x^{2}\cosh x-\frac{1%
}{4}x\sinh 3x+\frac{3}{4}x\sinh x  \notag \\
&=&\frac{1}{4}\sum_{n=0}^{\infty }\frac{3^{2n}}{\left( 2n\right) !}x^{2n+2}-%
\frac{1}{4}\sum_{n=0}^{\infty }\frac{1}{\left( 2n\right) !}x^{2n+2}  \notag
\\
&&-\frac{1}{4}\sum_{n=1}^{\infty }\frac{3^{2n-1}}{\left( 2n-1\right) !}%
x^{2n}+\frac{3}{4}\sum_{n=1}^{\infty }\frac{1}{\left( 2n-1\right) !}x^{2n} 
\notag \\
&=&\sum_{n=3}^{\infty }\frac{\left( 4n\left( n-2\right) 9^{n-1}-4n\left(
n-2\right) \right) }{4\left( 2n\right) !}x^{2n}:=\sum_{n=3}^{\infty }\frac{%
b_{n}}{4\left( 2n\right) !}x^{2n},  \label{Be}
\end{eqnarray}%
\begin{eqnarray*}
C\left( x\right) &=&-2x^{2}\cosh x+x\sinh x+\frac{1}{4}\cosh 3x-\frac{1}{4}%
\cosh x \\
&=&-2\sum_{n=0}^{\infty }\frac{1}{\left( 2n\right) !}x^{2n+2}+\sum_{n=1}^{%
\infty }\frac{1}{\left( 2n-1\right) !}x^{2n}+\frac{1}{4}\sum_{n=0}^{\infty }%
\frac{3^{2n}}{\left( 2n\right) !}x^{2n}-\frac{1}{4}\sum_{n=0}^{\infty }\frac{%
1}{\left( 2n\right) !}x^{2n}
\end{eqnarray*}%
\begin{equation}
=\sum_{n=3}^{\infty }\frac{\left( 9^{n}-32n^{2}+24n-1\right) }{4\left(
2n\right) !}x^{2n}:=\sum_{n=3}^{\infty }\frac{c_{n}}{4\left( 2n\right) !}%
x^{2n}.  \label{Ce}
\end{equation}

We see clearly that, by Lemma \ref{Lemma monotonicity of ratio}, if we can
prove $f_{2}\left( x\right) \geq (\leq )0$ for all $x\in (0,\infty )$ then $%
H_{p,q}$ defined by (\ref{H_p,q}) is increasing (decreasing) on $(0,\infty )$%
. To this end, we need to prove the following important statement.

\begin{lemma}
\label{Lemma f3}Let $f_{3}$ be defined on $(0,\infty )$ by%
\begin{equation}
f_{3}\left( x\right) =\frac{pA\left( x\right) -qB\left( x\right) }{C\left(
x\right) }+1.  \label{f3}
\end{equation}%
where $A\left( x\right) ,B\left( x\right) $ and $C\left( x\right) $ are
defined by (\ref{A}), (\ref{B}) and (\ref{C}), respectively. Then

(i) $f_{3}$ is strictly increasing on $(0,\infty )$ if $\left( p,q\right)
\in \mathbb{I}_{1}$, where 
\begin{equation}
\mathbb{I}_{1}=\left\{ q=0,p>0\right\} \cup \left\{ q>0,\tfrac{p}{q}\geq 
\tfrac{23}{17}\right\} \cup \left\{ q<0,\tfrac{p}{q}\leq 1\right\} ,
\label{I1}
\end{equation}%
and we have%
\begin{equation*}
\frac{5}{8}p-\frac{15}{8}q+1<f_{3}\left( x\right) <\infty ;
\end{equation*}

(ii) $f_{3}$ is strictly decreasing on $(0,\infty )$ if $\left( p,q\right)
\in \mathbb{I}_{2}$, where 
\begin{equation}
\mathbb{I}_{2}=\left\{ q=0,p<0\right\} \cup \left\{ q>0,\tfrac{p}{q}\leq
1\right\} \cup \left\{ q<0,\tfrac{p}{q}\geq \tfrac{23}{17}\right\} ,
\label{I2}
\end{equation}%
and we have%
\begin{equation*}
-\infty <f_{3}\left( x\right) <\frac{5}{8}p-\frac{15}{8}q+1.
\end{equation*}
\end{lemma}

\begin{proof}
Using (\ref{Ae}), (\ref{Be}) and (\ref{Ce}) gives 
\begin{equation*}
f_{3}\left( x\right) -1=\frac{pA\left( x\right) -qB\left( x\right) }{C\left(
x\right) }=\dfrac{\sum_{n=3}^{\infty }\frac{\left( pa_{n}-qb_{n}\right) }{%
4\left( 2n\right) !}x^{2n}}{\sum_{n=3}^{\infty }\frac{c_{n}}{4\left(
2n\right) !}x^{2n}},
\end{equation*}%
where%
\begin{eqnarray}
a_{n} &=&\left( \left( 4n^{2}-14n+9\right) 9^{n-1}+12n^{2}-10n-1\right) ,
\label{a_n} \\
b_{n} &=&\left( 4n\left( n-2\right) 9^{n-1}-4n\left( n-2\right) \right) ,
\label{b_n} \\
c_{n} &=&\left( 9^{n}-32n^{2}+24n-1\right) .  \label{c_n}
\end{eqnarray}%
In order to observe the monotonicity of $f_{3}$, we need to investigate the
monotonicity of series 
\begin{equation*}
\frac{\left( pa_{n}-qb_{n}\right) /\left( 4\left( 2n\right) !\right) }{%
c_{n}/\left( 4\left( 2n\right) !\right) }=\frac{pa_{n}-qb_{n}}{c_{n}}.
\end{equation*}%
We have%
\begin{eqnarray*}
&&\frac{pa_{n+1}-qb_{n+1}}{c_{n+1}}-\frac{pa_{n}-qb_{n}}{c_{n}} \\
&=&\frac{p\left( a_{n+1}c_{n}-a_{n}c_{n+1}\right) -q\left(
b_{n+1}c_{n}-b_{n}c_{n+1}\right) }{c_{n}c_{n+1}} \\
&=&\frac{pv_{n}-qu_{n}}{c_{n}c_{n+1}}=\left\{ 
\begin{array}{cc}
p\frac{v_{n}}{c_{n}c_{n+1}} & \text{if }q=0, \\ 
\dfrac{v_{n}}{c_{n}c_{n+1}}q\left( \dfrac{p}{q}-\dfrac{u_{n}}{v_{n}}\right)
& \text{if }q\neq 0,%
\end{array}%
\right.
\end{eqnarray*}%
where%
\begin{eqnarray*}
u_{n} &=&b_{n+1}c_{n}-b_{n}c_{n+1} \\
&=&2\times 9^{n-1}\left[ \left( 36n-18\right) 9^{n}-\left(
512n^{4}-384n^{3}-560n^{2}+792n-36\right) \right] \\
&&+4\left( 42n+40n^{2}-1\right) ,
\end{eqnarray*}%
\begin{eqnarray*}
v_{n} &=&a_{n+1}c_{n}-a_{n}c_{n+1} \\
&=&2\times 9^{n-1}\left[ \left( 36n-45\right) 9^{2n}-\left(
512n^{4}-1152n^{3}+1072n^{2}\right) \right] \\
&&+2\times \left[ 2\left( 28n+5\right) 9^{n}-\left( 16n^{2}+60n+5\right) %
\right] \\
&:&=2\times 9^{n-1}v_{n}^{\prime }+2v_{n}^{\prime \prime }.
\end{eqnarray*}%
Now we distinguish three cases to discuss the monotonicity of $f_{3}$.

(i) When $q=0$, we have $c_{n},v_{n}>0$ for $n\geq 3$. Indeed, we use
binomial expansion to get%
\begin{eqnarray*}
c_{n} &=&9^{n}-32n^{2}+24n-1=\left( 1+8\right) ^{n}-32n^{2}+24n-1 \\
&>&1+8n+\tfrac{n\left( n-1\right) }{2}8^{2}-32n^{2}+24n-1=0.
\end{eqnarray*}%
Application of binomial expansion again we have%
\begin{eqnarray*}
v_{n}^{\prime } &=&\left( 36n-45\right) \left( 1+8\right) ^{n}-\left(
512n^{4}-1152n^{3}+1072n^{2}\right) \\
&>&\left( 36n-45\right) \left( 1+8n+\tfrac{n\left( n-1\right) }{2}8^{2}+%
\tfrac{n\left( n-1\right) \left( n-2\right) }{6}8^{3}\right) -\left(
512n^{4}-1152n^{3}+1072n^{2}\right) \\
&=&2560n^{4}-10752n^{3}+14288n^{2}-6564n-45 \\
&=&2560\left( n-3\right) ^{4}+19968\left( n-3\right) ^{3}+55760\left(
n-3\right) ^{2}+65340\left( n-3\right) +25911>0,
\end{eqnarray*}%
\begin{eqnarray*}
v_{n}^{\prime \prime } &=&2\left( 28n+5\right) \left( 1+8\right) ^{n}-\left(
16n^{2}+60n+5\right) \\
&>&2\left( 28n+5\right) \left( 1+8n\right) -\left( 16n^{2}+60n+5\right) \\
&=&432n^{2}+76n+5>0
\end{eqnarray*}%
which show that $v_{n}=v_{n}^{\prime }+v_{n}^{\prime \prime }>0$ for $n\geq
3 $. Thus, $\left( pa_{n}-qb_{n}\right) /c_{n}$ is increasing if $p\geq 0$
and decreasing if $p<0$, and by Lemma \ref{Lemma monotocity of A/B} so is $%
f_{3}-1$ on $(0,\infty )$. Hence, we have%
\begin{eqnarray*}
\frac{5}{8}p+1 &=&\lim_{x\rightarrow 0^{+}}f_{3}\left( x\right) <f_{3}\left(
x\right) <\lim_{x\rightarrow \infty }f_{3}\left( x\right) =\infty \text{ if }%
p>0, \\
f_{3}\left( x\right) &=&1\text{ if }p=0, \\
-\infty &=&\lim_{x\rightarrow \infty }f_{3}\left( x\right) <f_{3}\left(
x\right) <\lim_{x\rightarrow 0^{+}}f_{3}\left( x\right) =\frac{5}{8}p+1\text{
if }p<0.
\end{eqnarray*}

(ii) When $q\neq 0$, we claim that $u_{n}/v_{n}$ is decreasing for $n\geq 3$%
. Since $v_{n}>0$ for $n\geq 3$, it suffices to show that $%
u_{n}v_{n+1}-u_{n+1}v_{n}>0$. Factoring and arranging give us to%
\begin{eqnarray*}
\frac{u_{n}v_{n+1}-u_{n+1}v_{n}}{c_{n+1}} &=&a_{n+2}\left(
c_{n}b_{n+1}-b_{n}c_{n+1}\right) +b_{n+2}\left(
a_{n}c_{n+1}-c_{n}a_{n+1}\right) \\
+c_{n+2}\left( b_{n}a_{n+1}-a_{n}b_{n+1}\right) &=&\frac{16}{3}w_{n},
\end{eqnarray*}%
where%
\begin{eqnarray}
w_{n} &=&9^{3n+2}-\left( 1024n^{4}-2560n^{3}+2752n^{2}+243\right) 9^{2n} 
\notag \\
&&+\left( 1024n^{4}+2560n^{3}+2752n^{2}+243\right) 9^{n}-81.  \label{w_n}
\end{eqnarray}%
As shown previously, $c_{n+1}>0$ for $n\geq 3$, and we only need to prove $%
w_{n}>0$ for $n\geq 3$. Since the sum of the third and fourth terms in (\ref%
{w_n}) is obviously positive, and it suffices to show that the sum of the
first and second is also positive. Using binomial expansion again, we have%
\begin{eqnarray*}
9^{-2n}w_{n} &>&9^{n+2}-\left( 1024n^{4}-2560n^{3}+2752n^{2}+243\right) , \\
&>&1+8(n+2)+\tfrac{\left( n+2\right) \left( n+1\right) }{2}8^{2}+\tfrac{%
\left( n+2\right) \left( n+1\right) n}{6}8^{3}+\tfrac{\left( n+2\right)
\left( n+1\right) n\left( n-1\right) }{24}8^{4} \\
&&+\tfrac{\left( n+2\right) \left( n+1\right) n\left( n-1\right) \left(
n-2\right) }{120}8^{5}-\left( 1024n^{4}-2560n^{3}+2752n^{2}+243\right) \\
&=&\frac{2}{15}\left(
2048n^{5}-6400n^{4}+12160n^{3}-19\,760n^{2}+7692n-1215\right) \\
&=&2048\left( n-3\right) ^{5}+24320\left( n-3\right) ^{4}+119680\left(
n-3\right) ^{3} \\
&&+297040\left( n-3\right) ^{2}+355692\left( n-3\right) +151605 \\
&>&0,
\end{eqnarray*}%
which proves the decreasing property of $u_{n}/v_{n}$ for $n\geq 3$. It
follows that%
\begin{equation*}
1=\lim_{n\rightarrow \infty }\dfrac{u_{n}}{v_{n}}<\dfrac{u_{n}}{v_{n}}\leq 
\frac{u_{3}}{v_{3}}=\frac{23}{17},
\end{equation*}%
and then, we conclude that%
\begin{equation*}
\tfrac{pa_{n+1}-qb_{n+1}}{c_{n+1}}-\tfrac{pa_{n}-qb_{n}}{c_{n}}=\dfrac{v_{n}%
}{c_{n}c_{n+1}}q\left( \tfrac{p}{q}-\tfrac{u_{n}}{v_{n}}\right) \left\{ 
\begin{array}{ll}
>0\bigskip & \text{if }q>0,\frac{p}{q}\geq \frac{23}{17}, \\ 
<0\bigskip & \text{if }q<0,\frac{p}{q}\geq \frac{23}{17}, \\ 
<0\bigskip & \text{if }q>0,\frac{p}{q}\leq 1, \\ 
>0\bigskip & \text{if }q<0,\frac{p}{q}\leq 1,%
\end{array}%
\right.
\end{equation*}%
which by Lemma \ref{Lemma monotocity of A/B} yield the desired monotonicity
results. And, easy calculations gives%
\begin{eqnarray*}
\lim_{x\rightarrow 0^{+}}f_{3}\left( x\right) -1 &=&\frac{5}{8}p-\frac{15}{8}%
q, \\
\lim_{x\rightarrow \infty }f_{3}\left( x\right) -1 &=&\left\{ 
\begin{array}{cc}
\infty \bigskip & \text{if }q>0,\frac{p}{q}\geq \frac{23}{17}\text{ or }q<0,%
\frac{p}{q}\leq 1, \\ 
-\infty & \text{if }q<0,\frac{p}{q}\geq \frac{23}{17}\text{ or }q>0,\frac{p}{%
q}\leq 1.%
\end{array}%
\right.
\end{eqnarray*}%
Thus the proof is completed
\end{proof}

From the lemma above, we easily get the monotonicity of $H_{p,q}$.

\begin{proposition}
\label{P main1}Let $H_{p,q}$ be defined on $(0,\infty $ by (\ref{H_p,q}).
Then

(i) $H_{p,q}$ is increasing on $(0,\infty )$ if%
\begin{equation*}
\left( p,q\right) \in (\mathbb{I}_{1}\cup \left\{ \left( 0,0\right) \right\}
)\cap \left\{ \tfrac{5}{8}p-\tfrac{15}{8}q+1\geq 0\right\} ,
\end{equation*}%
where $\mathbb{I}_{1}$ is defined by (\ref{I1});

(ii) $H_{p,q}$ is decreasing on $(0,\infty )$ if%
\begin{equation*}
\left( p,q\right) \in \mathbb{I}_{2}\cap \left\{ \tfrac{5}{8}p-\tfrac{15}{8}%
q+1\leq 0\right\} ,
\end{equation*}%
where $\mathbb{I}_{2}$ is defined by (\ref{I2}).
\end{proposition}

\begin{proof}
As mentioned previously, to prove the monotonicity of $H_{p,q}$, it suffices
to deal with the sings of $f_{2}\left( x\right) $ on $\left( 0,\infty
\right) $. It is clear that%
\begin{equation*}
\frac{f_{2}\left( x\right) }{C\left( x\right) }=f_{3}\left( x\right) ,
\end{equation*}%
where $f_{3}\left( x\right) $ is defined by (\ref{f3}). Then, $\func{sgn}%
f_{2}\left( x\right) =\func{sgn}f_{3}\left( x\right) $ due to $C\left(
x\right) >0$ for $x\in (0,\infty )$.

(i) When $\left( p,q\right) \in \mathbb{I}_{1}\cup \left\{ \left( 0,0\right)
\right\} $, it is obtained from Lemma \ref{Lemma f3} that $f_{2}\left(
x\right) >0$ for $x\in \left( 0,\infty \right) $ provided $%
\inf_{x>0}f_{3}\left( x\right) =5p/8-15q/8+1\geq 0$. Utilizing the relation (%
\ref{df1})\ and Lemma \ref{Lemma monotonicity of ratio} we get the
conclusion that $H_{p,q}$ is increasing on $(0,\infty )$ for $\left(
p,q\right) \in \left( \mathbb{I}_{1}\cup \left\{ \left( 0,0\right) \right\}
\right) \cap \left\{ 5p/8-15q/8+1\geq 0\right\} $.

(ii) When $\left( p,q\right) \in \mathbb{I}_{2}$, $f_{2}\left( x\right) <0$
for $x\in \left( 0,\infty \right) $ so long as $\sup_{x>0}f_{3}\left(
x\right) =5p/8-15q/8+1\leq 0$. Then, $H_{p,q}$ is decreasing on $(0,\infty )$
for $\left( p,q\right) \in \mathbb{I}_{2}\cap \left\{ 5p/8-15q/8+1\leq
0\right\} $.

Thus we complete the proof.
\end{proof}

It is easy to check that $\left( p,q\right) \in (\mathbb{I}_{1}\cup \left\{
\left( 0,0\right) \right\} \cap \left\{ \tfrac{5}{8}p-\tfrac{15}{8}q+1\geq
0\right\} $ is equivalent to%
\begin{equation*}
p\geq \left\{ 
\begin{array}{ll}
3q-\frac{8}{5}\bigskip & \text{if }q\in \lbrack \frac{34}{35},\infty ) \\ 
\tfrac{23}{17}q\bigskip & \text{if }q\in \lbrack 0,\frac{34}{35}), \\ 
q & \text{if }q\in \left( -\infty ,0\right) ,%
\end{array}%
\right.
\end{equation*}%
while $\left( p,q\right) \in \mathbb{I}_{2}\cap \left\{ \tfrac{5}{8}p-\tfrac{%
15}{8}q+1\leq 0\right\} $ is equivalent to%
\begin{equation*}
p\leq \left\{ 
\begin{array}{ll}
q\bigskip & \text{if }q\in \lbrack \frac{4}{5},\infty ), \\ 
3q-\frac{8}{5} & \text{if }q\in \left( -\infty ,\frac{4}{5}\right) .%
\end{array}%
\right.
\end{equation*}%
Then Proposition \ref{P main1} can be restated as follows.

\begin{proposition}
\label{P main2}Let $H_{p,q}$ be defined on $(0,\infty $ by (\ref{H_p,q}).
Then

(i) when $q\in \left[ 34/35,\infty \right) $, $H_{p,q}$ is increasing on $%
(0,\infty )$ for $p\geq 3q-8/5$ and decreasing for $p\leq q$;

(ii) when $q\in \left[ 4/5,34/35\right) $, $H_{p,q}$ is increasing on $%
(0,\infty )$ for $p\geq 23q/17$ and decreasing for $p\leq q$;

(iii) when $q\in (0,4/5$, $H_{p,q}$ is increasing on $(0,\infty )$ for $%
p\geq 23q/17$ and decreasing for $p\leq 3q-8/5$;

(iv) when $q\in (-\infty ,0]$, $H_{p,q}$ is increasing on $(0,\infty )$ for $%
p\geq q$ and decreasing for $p\leq 3q-8/5$.
\end{proposition}

On the other hand, $\left( p,q\right) \in (\mathbb{I}_{1}\cup \left\{ \left(
0,0\right) \right\} \cap \left\{ \tfrac{5}{8}p-\tfrac{15}{8}q+1\geq
0\right\} $ is equivalent to%
\begin{equation*}
q\leq \left\{ 
\begin{array}{ll}
\tfrac{p}{3}+\tfrac{8}{15}\bigskip & \text{if }p\in \lbrack \frac{46}{35}%
,\infty ) \\ 
\tfrac{17}{23}p\bigskip & \text{if }q\in \lbrack 0,\frac{46}{35}), \\ 
p & \text{if }q\in \left( -\infty ,0\right) ,%
\end{array}%
\right.
\end{equation*}%
while $\left( p,q\right) \in \mathbb{I}_{2}\cap \left\{ \tfrac{5}{8}p-\tfrac{%
15}{8}q+1\leq 0\right\} $ is equivalent to%
\begin{equation*}
q\geq \left\{ 
\begin{array}{ll}
p\bigskip & \text{if }p\in \lbrack \frac{4}{5},\infty ), \\ 
\tfrac{p}{3}+\tfrac{8}{15} & \text{if }q\in (-\infty ,\frac{4}{5}).%
\end{array}%
\right.
\end{equation*}%
Then Proposition \ref{P main1} also can be restated in another equivalent
assertion.

\begin{proposition}
\label{P main3}Let $H_{p,q}$ be defined on $(0,\infty $ by (\ref{H_p,q}).
Then

(i) when $p\in \left[ 46/35,\infty \right) $, $H_{p,q}$ is increasing on $%
(0,\infty )$ for $q\leq p/3+8/15$ and decreasing for $q\geq p$;

(ii) when $p\in \left[ 4/5,46/35\right) $, $H_{p,q}$ is increasing on $%
(0,\infty )$ for $q\leq 17p/23$ and decreasing for $q\geq p$;

(iii) when $p\in (0,4/5)$, $H_{p,q}$ is increasing on $(0,\infty )$ for $%
q\leq 17p/23$ and decreasing for $q\geq p/3+8/15$;

(iv) when $p\in (-\infty ,0]$, $H_{p,q}$ is increasing on $(0,\infty )$ for $%
q\leq p$ and decreasing for $q\geq p/3+8/15$.
\end{proposition}

Put $p=kq$, then by Proposition \ref{P main1} in combination with its proof,
we have

\begin{corollary}
\label{Corollary p=kq}Let $H_{p,q}$ be defined on $(0,\infty )$ by (\ref%
{H_p,q}). Then

(i) when $k\in \left( 3,\infty \right) $, $H_{kq,q}$ is increasing for $%
q\geq 0$ and decreasing for $q\leq 8/\left( 5\left( 3-k\right) \right) $;

(ii) when $k=3$, $H_{kq,q}$ is increasing for $q\in \mathbb{R}$;

(iii) when $k\in \left[ 23/17,3\right) $, $H_{kq,q}$ is increasing for $%
0\leq q\leq 8/\left( 5\left( 3-k\right) \right) $;

(iv) when $k\in (1,23/17)$, $H_{kq,q}$ is increasing for $q=0$;

(v) when $k\in (-\infty ,1]$, $H_{kq,q}$ is increasing for $q\leq 0$ and
decreasing for $q\geq 8/\left( 5\left( 3-k\right) \right) $.
\end{corollary}

If $5p/8-15q/8+1=0$, that is, $p=3q-8/5$ or $q=p/3+8/15$, then we easily
check that 
\begin{eqnarray*}
&&(\mathbb{I}_{1}\cup \left\{ \left( 0,0\right) \right\} \cap \left\{ \tfrac{%
5}{8}p-\tfrac{15}{8}q+1=0\right\} \\
&=&\left\{ q\geq \tfrac{34}{35},p=3q-\tfrac{8}{5}\right\} =\left\{ p\geq 
\tfrac{46}{35},q=\tfrac{p}{3}+\tfrac{8}{15}\right\} ,
\end{eqnarray*}%
\begin{eqnarray*}
&&\mathbb{I}_{2}\cap \left\{ \tfrac{5}{8}p-\tfrac{15}{8}q+1=0\right\} \\
&=&\left\{ q\leq \tfrac{4}{5},p=3q-\tfrac{8}{5}\right\} =\left\{ p\leq 
\tfrac{4}{5},q=\tfrac{p}{3}+\tfrac{8}{15}\right\} \text{,}
\end{eqnarray*}%
and then by Proposition \ref{P main1} we get

\begin{corollary}
\label{Corollary p=3q-8/5}Let $H_{p,q}$ be defined on $(0,\infty $ by (\ref%
{H_p,q}). Then $H_{3q-8/5,q}$ is increasing if $q\geq 34/35$ and decreasing
if $q\leq 4/5$. In other words, $H_{p,p/3+8/15}$ is increasing if $p\geq
46/35$ and decreasing if $p\leq 4/5$.
\end{corollary}

\section{Results}

In this section, we will give some new inequalities involving hyperbolic
functions by using monotonicity theorems given in previous section.

Note that%
\begin{eqnarray*}
&&(\mathbb{I}_{1}\cup \left\{ \left( 0,0\right) \right\} \cap \left\{ \tfrac{%
5}{8}p-\tfrac{15}{8}q+1\geq 0\right\} \\
&=&\left\{ 0\leq q\leq \min (17p/23,p/3+8/15)\right\} \cup \left\{ q\leq
\min (0,p,p/3+8/15\right\} ),
\end{eqnarray*}%
\begin{eqnarray*}
&&\mathbb{I}_{2}\cap \left\{ \tfrac{5}{8}p-\tfrac{15}{8}q+1\leq 0\right\} \\
&=&\left\{ \max (17p/23,p/3+8/15)\leq q<0\right\} \cup \left\{ q\geq \max
(0,p,p/3+8/15)\right\}
\end{eqnarray*}%
and that $H_{p,q}\left( 0^{+}\right) <\left( >\right) H_{p,q}\left(
0^{+}\right) =1/3$ is equivalent to $Sh_{p}\left( x\right) <\left( >\right)
\left( 1/3\right) Ch_{q}\left( x\right) $ for $x\in (0,\infty )$. By
Proposition \ref{P main1}, we obtain the following theorem immediately.

\begin{theorem}
\label{MT1}(i) If $0\leq q\leq \min (17p/23,p/3+8/15)$ or $q\leq \min
(0,p,p/3+8/15)$, then the inequalities%
\begin{eqnarray}
\frac{\left( \tfrac{\sinh x}{x}\right) ^{p}-1}{p} &>&\frac{1}{3}\frac{\cosh
^{q}x-1}{q}\text{ if }pq\neq 0,  \label{MI1} \\
\ln \frac{\sinh x}{x} &>&\frac{1}{3}\frac{\cosh ^{q}x-1}{q}\text{ if }%
p=0,q\neq 0,  \label{MI1p=0} \\
\frac{\left( \tfrac{\sinh x}{x}\right) ^{p}-1}{p} &>&\frac{1}{3}\ln \cosh x%
\text{ \ \ \ if }p\neq 0,q=0,  \label{Mi1q=0} \\
\ln \frac{\sinh x}{x} &>&\frac{1}{3}\ln \cosh x\text{ \ \ \ if }p=q=0,
\label{MI1p=q=0}
\end{eqnarray}%
hold for $x\in (0,\infty )$, where $1/3$ is the best constant.

(ii) If $\max (17p/23,p/3+8/15)\leq q<0$ or $q\geq \max (0,p,p/3+8/15)$,
then (\ref{MI1}), (\ref{MI1p=0}) and (\ref{Mi1q=0}) are reversed.
\end{theorem}

For clarity of expressions, in what follows we will directly write $%
Sh_{p}\left( x\right) ,Ch_{q}\left( x\right) ,H_{p,q}\left( x\right) $ etc.
by their general formulas, and if $pq=0$, then we regard them as limits at $%
p=0$ or $q=0$, unless otherwise specified. Now we are ready to establish
sharp inequalities for hyperbolic by Propositions \ref{P main2} and \ref{P
main3}, Corollaries \ref{Corollary p=kq} and \ref{Corollary p=3q-8/5}. To
this end, we need a lemma.

\begin{lemma}
\label{Lemma D_p.q}Let $D_{p,q}$ be defined on $(0,\infty )$ by 
\begin{equation}
D_{p,q}\left( x\right) =Sh_{p}\left( x\right) -\frac{1}{3}Ch_{q}\left(
x\right) =\frac{\left( \frac{\sinh x}{x}\right) ^{p}-1}{p}-\frac{\left(
\cosh x\right) ^{q}-1}{3q}.  \label{D_p,q}
\end{equation}%
(i) We have%
\begin{eqnarray}
\lim_{x\rightarrow 0^{+}}\frac{D_{p,q}\left( x\right) }{x^{4}} &=&\frac{1}{72%
}\left( p-3q+\frac{8}{5}\right) ,  \label{Limit1} \\
\lim_{x\rightarrow 0^{+}}\frac{D_{3q-8/5,q}\left( x\right) }{x^{6}} &=&\frac{%
1}{270}\left( q-\frac{34}{35}\right) .  \label{Limit1 p=3q-8/5}
\end{eqnarray}%
(ii) For $p,q\geq 0$, we have%
\begin{equation}
\lim_{x\rightarrow \infty }e^{-qx}D_{p,q}\left( x\right) =\left\{ 
\begin{array}{ll}
\infty \medskip & \text{if }p>q\geq 0, \\ 
\infty \medskip & \text{if }p\geq q=0, \\ 
-\frac{2^{-q}}{3q}\medskip & \text{if }q\geq p>0, \\ 
-\frac{2^{-q}}{3q} & \text{if }q>p=0;%
\end{array}%
\right.  \label{Limit2p,q>=0}
\end{equation}%
for other cases, we have%
\begin{equation}
\lim_{x\rightarrow \infty }D_{p,q}\left( x\right) =\left\{ 
\begin{array}{ll}
\infty \medskip & \text{if }p\geq 0,q<0, \\ 
-\infty \medskip & \text{if }p<0,q\geq 0, \\ 
\frac{1}{3q}-\frac{1}{p} & \text{if }p<0,q<0.%
\end{array}%
\right.  \label{Limit2others}
\end{equation}
\end{lemma}

\begin{proof}
(i) For $pq\neq 0$, expanding in power series yields%
\begin{eqnarray}
D_{p,q}\left( x\right) &=&\tfrac{\left( \frac{\sinh x}{x}\right) ^{p}-1}{p}-%
\tfrac{\left( \cosh x\right) ^{q}-1}{3q}  \notag \\
&=&\tfrac{5p-15q+8}{360}x^{4}+\tfrac{35p^{2}-42p-315q^{2}+630q-320}{45360}%
x^{6}+o\left( x^{8}\right) ,  \label{D e}
\end{eqnarray}%
which leads to (\ref{Limit1}). It is easy to check that it holds for $p=0$
or $q=0$.

If $p=3q-8/5$, then we have%
\begin{equation*}
D_{p,q}\left( x\right) =\frac{35q-34}{9450}x^{6}+o\left( x^{8}\right) ,
\end{equation*}%
which implies (\ref{Limit1 p=3q-8/5}).

(ii) For $p,q>0$, we have%
\begin{eqnarray*}
\frac{D_{p,q}\left( x\right) }{e^{qx}} &=&\frac{e^{-qx}}{p}\left( \frac{%
e^{x}-e^{-x}}{2x}\right) ^{p}-\frac{e^{-qx}}{p}-\frac{e^{-qx}}{3q}\left( 
\frac{e^{x}+e^{-x}}{2}\right) ^{q}+\frac{e^{-qx}}{3q} \\
&=&\frac{1}{p}\frac{e^{\left( p-q\right) x}}{x^{p}}\left( \frac{1-e^{-2x}}{2}%
\right) ^{p}-\frac{1}{3q}\left( \frac{1+e^{-2x}}{2}\right) ^{q}-\left( \frac{%
1}{p}-\frac{1}{3q}\right) e^{-qx} \\
&\rightarrow &\left\{ 
\begin{array}{cc}
\medskip \infty & \text{if }p>q>0, \\ 
-\frac{2^{-q}}{3q} & \text{if }q\geq p>0,%
\end{array}%
\right. \text{ as }x\rightarrow \infty \text{;}
\end{eqnarray*}%
for $p=0,q>0$, we have%
\begin{eqnarray*}
e^{-qx}D_{0,q}\left( x\right) &=&e^{-qx}\ln \frac{e^{x}-e^{-x}}{2x}-\frac{%
e^{-qx}}{3q}\left( \frac{e^{x}+e^{-x}}{2}\right) ^{q}+\frac{e^{-qx}}{3q} \\
&=&xe^{-qx}+e^{-qx}\ln \frac{1-e^{-2x}}{2}-e^{-qx}\ln x-\frac{1}{3q}\left( 
\frac{1+e^{-2x}}{2}\right) ^{q}+\frac{e^{-qx}}{3q} \\
&\rightarrow &-\frac{2^{-q}}{3q}\text{, as }x\rightarrow \infty \text{;}
\end{eqnarray*}%
for $p=0,q=0$, we have%
\begin{eqnarray*}
D_{0,0}\left( x\right) &=&\ln \frac{e^{x}-e^{-x}}{2x}-\frac{1}{3}\ln \frac{%
e^{x}+e^{-x}}{2} \\
&=&x+\ln \frac{1-e^{-2x}}{2}-\ln x-\frac{x}{3}-\frac{1}{3}\frac{1+e^{-2x}}{2}
\\
&=&x\left( \frac{2}{3}-\frac{\ln x}{x}\right) +\ln \frac{1-e^{-2x}}{2}-\frac{%
1}{3}\frac{1+e^{-2x}}{2} \\
&\rightarrow &\infty \text{, as }x\rightarrow \infty \text{;}
\end{eqnarray*}%
for $p>0,q=0$, utilizing the increasing property of $Sh_{p}\left( x\right)
=U_{p}\left( \left( \sinh x\right) /x\right) $, we get $Sh_{p}\left(
x\right) >Sh_{0}\left( x\right) $, and then, $\lim_{x\rightarrow \infty
}D_{p,0}\left( x\right) =\lim_{x\rightarrow \infty }D_{0,0}\left( x\right)
=\infty $, which gives $\lim_{x\rightarrow \infty }D_{p,0}\left( x\right)
=\infty $.

To sum up, relation (\ref{Limit2p,q>=0}) hold.

While (\ref{Limit2others}) follows from the fact that for $t>1$%
\begin{equation*}
U_{p}\left( \infty \right) =\lim_{t\rightarrow \infty }\tfrac{t^{p}-1}{p}%
=\infty \text{ if }p\geq 0,U_{p}\left( \infty \right) =\lim_{t\rightarrow
\infty }\tfrac{t^{p}-1}{p}=-\frac{1}{p}\text{ if }p<0,
\end{equation*}%
which proves the lemma.
\end{proof}

Utilizing Proposition \ref{P main2} and lemma above we have the following
theorem.

\begin{theorem}
\label{MT2}Let $x\in (0,\infty )$. Then

(i) when $q\in \left[ 34/35,\infty \right) $, the double inequality%
\begin{equation}
\frac{\left( \tfrac{\sinh x}{x}\right) ^{p_{2}}-1}{p_{2}}<\frac{\cosh ^{q}x-1%
}{3q}<\frac{\left( \tfrac{\sinh x}{x}\right) ^{p_{1}}-1}{p_{1}}  \tag{MI2}
\label{MI2}
\end{equation}%
holds if and only if $p_{1}\geq 3q-8/5$ and $p_{2}\leq q$;

(ii) when $q\in \left[ 4/5,34/35\right) $, the double inequality (\ref{MI2})
holds for $p_{1}\geq 23q/17$ and if and only if $p_{2}\leq q$;

(iii) when $q\in (0,4/5)$, the double inequality (\ref{MI2}) holds for $%
p_{1}\geq 23q/17$ and if and only if $p_{2}\leq 3q-8/5$;

(iv) when $q\in (-\infty ,0]$, the double inequality (\ref{MI2}) holds if
and only if $p_{1}\geq q$ and $p_{2}\leq 3q-8/5$.
\end{theorem}

\begin{proof}
The sufficiencies in the cases of (i)--(iv) are due to Proposition \ref{P
main2}. Now we show the necessities in certain cases.

(i) When $q\in \left[ 34/35,\infty \right) $, the condition $p_{1}\geq
3q-8/5 $ is necessary for the second inequality in (\ref{MI2}) to hold. If
the second in (\ref{MI2}) holds, then we have $\lim_{x\rightarrow
0^{+}}x^{-4}D_{p_{1},q}\left( x\right) \geq 0$, which, by (\ref{Limit1}),
yields $p_{1}\geq 3q-8/5$. We claim that the condition $p_{2}\leq q$ is also
necessary for the first inequality in (\ref{MI2}) to be true. If there is a $%
p_{2}>q\in \left[ 34/35,\infty \right) $ such that the first inequality in (%
\ref{MI2}) holds. then by (\ref{Limit2others}) there must be $%
\lim_{x\rightarrow \infty }e^{-qx}D_{p_{2},q}\left( x\right) =\infty $,
which yields a contradiction. Hence, the condition $p_{2}>q$ is also
necessary.

(ii) When $q\in \left[ 4/5,34/35\right) $, similar to part two of proof (i),
the condition $p_{2}\leq q$ is necessary for the first inequality in (\ref%
{MI2}) to be valid.

(iii) When $q\in \left[ 0,4/5\right) $, in the same way as part one of proof
(i), the condition $p_{2}\leq 3q-8/5$ is necessary for the first inequality
in (\ref{MI2}) to hold.

(iv) When $q\in (-\infty ,0)$, analogous to the case of $q\in \left[
34/35,\infty \right) $, we can prove the conditions $p_{1}\geq q$ and $%
p_{2}\leq 3q-8/5$ are necessary.

This completes the proof.
\end{proof}

\begin{remark}
Taking $k=1$ in Theorem \ref{MT2}, we get a equivalent result of Theorem
Zhu1.
\end{remark}

Similarly, by Proposition \ref{P main3} and Lemma \ref{Lemma D_p.q} we can
prove the following statement.

\begin{theorem}
\label{MT3}Let $x\in (0,\infty )$. Then

(i) when $p\in \left[ 46/35,\infty \right) $, the double inequality%
\begin{equation}
\frac{\cosh ^{q_{1}}x-1}{3q_{1}}<\frac{\left( \tfrac{\sinh x}{x}\right)
^{p}-1}{p}<\frac{\cosh ^{q_{2}}x-1}{3q_{2}}  \tag{MI3}  \label{MI3}
\end{equation}%
holds if and only if $q_{1}\leq p/3+8/15$ and $q_{2}\geq p$;

(ii) when $p\in \left[ 4/5,46/35\right) $, the double inequality (\ref{MI3})
holds for $q_{1}\leq 17p/23$ and if and only if $q_{2}\geq p$;

(iii) when $p\in (0,4/5)$, the double inequality (\ref{MI3}) holds for $%
q_{1}\leq 17p/23$ and if and only if $q_{2}\geq p/3+8/15$;

(iv) when $p\in (-\infty ,0]$, the double inequality (\ref{MI3}) holds if
and only if $q_{1}\leq p$ and $q_{2}\geq p/3+8/15$.
\end{theorem}

\begin{remark}
\label{Remark T-M}(i) For $x\in (0,\infty )$, $Sh_{p}\left( x\right) <\left(
>\right) \left( 1/3\right) Ch_{q}\left( x\right) $ is equivalent to $\left(
\sinh x\right) /x>\left( <\right) M\left( \cosh x;p,q\right) $ for certain $%
\left( p,q\right) \in \Omega _{p,q}$, where 
\begin{equation}
M\left( t;p,q\right) =\left\{ 
\begin{array}{ll}
\bigskip \left( 1-\frac{p}{3q}+\frac{p}{3q}t^{q}\right) ^{1/p} & \text{if }%
pq\neq 0,\left( p,q\right) \in \Omega _{p,q}, \\ 
\bigskip \exp \frac{t^{q}-1}{3q} & \text{if }p=0,q\neq 0, \\ 
\bigskip \left( \frac{p}{3}\ln t+1\right) ^{1/p} & \text{if }p>0,q=0, \\ 
t^{1/3} & \text{if }p=q=0,%
\end{array}%
\right.  \label{M}
\end{equation}%
here $t=\cosh x\in \left( 1,\infty \right) $. It is easy to verify that for $%
t\in \left( 1,\infty \right) $, the largest set of $\left( p,q\right) $ such
that $M\left( t;p,q\right) $ exits in real number field is%
\begin{equation}
\Omega _{p,q}=\{\left( p,q\right) :p\geq 0\text{ or }3q\leq p\leq 0\}.
\label{Opq}
\end{equation}%
(ii) We suggest that $M$ is decreasing in $p$ and increasing in $q$ if $%
\left( p,q\right) \in \Omega _{p,q}$.

In fact, for $\left( p,q\right) \in \Omega _{p,q}$ with $pq\neq 0$,
logarithmic differentiation yields%
\begin{eqnarray*}
\frac{\partial \ln M}{\partial p} &=&\frac{1}{p^{2}}\left( -\ln \left( 1-%
\frac{p}{3q}+\frac{p}{3q}t^{q}\right) -\frac{p\left( 1-t^{q}\right) }{\left(
pt^{q}+3q-p\right) }\right) :=\frac{M_{1}\left( t;p,q\right) }{p^{2}}, \\
\frac{\partial M_{1}}{\partial p} &=&-\frac{p\left( 1-t^{q}\right) ^{2}}{%
\left( pt^{q}+3q-p\right) ^{2}},
\end{eqnarray*}%
which implies that $M_{1}$ is decreasing in $p$ on $(0,\infty )$ and
increasing on $\left( -\infty ,0\right) $. Hence we have $M_{1}\left(
t;p,q\right) <M_{1}\left( t;0,q\right) =0$, which means that $M$ is
decreasing in $p$.

It is easy to check that the monotonicity result of $M$ in $p$ is also true
for $pq=0$.

Similarly, we have%
\begin{equation*}
\frac{\partial \ln M}{\partial q}=-\frac{t^{q}\left( \ln
t^{-q}-t^{-q}+1\right) }{3q^{2}\left( \frac{p}{3q}t^{q}+1-\frac{p}{3q}%
\right) }>0,
\end{equation*}%
where the inequality holds due to $\ln x\leq x-1$ for $x>0$ and $\left(
p/\left( 3q\right) \right) t^{q}+1-\left( p/\left( 3q\right) \right) >0$ for 
$\left( t,p,q\right) \in (1,\infty )\times \Omega _{p,q}$, which proves the
monotonicity of $M$ with respect to $q$.
\end{remark}

\begin{remark}
By Remark above, if we add the condition that "$\left( p,q\right) \in \Omega
_{p,q}$" in Theorems \ref{MT2} and \ref{MT3}, and replace (\ref{MI2}), (\ref%
{MI3}) with%
\begin{equation}
\left( 1-\tfrac{p_{1}}{3q}+\tfrac{p_{1}}{3q}\cosh ^{q}x\right) ^{1/p_{1}}<%
\frac{\sinh x}{x}<\left( 1-\tfrac{p_{2}}{3q}+\tfrac{p_{2}}{3q}\cosh
^{q}x\right) ^{1/p_{2}},  \tag{MI2*}  \label{MI2*}
\end{equation}%
\begin{equation}
\left( 1-\tfrac{p}{3q_{1}}+\tfrac{p}{3q_{1}}\cosh ^{q_{1}}x\right) ^{1/p}<%
\frac{\sinh x}{x}<\left( 1-\tfrac{p}{3q_{2}}+\tfrac{p}{3q_{2}}\cosh
^{q_{2}}x\right) ^{1/p},  \tag{MI3*}  \label{MI3*}
\end{equation}%
respectively, then the two theorems are still true.
\end{remark}

Taking $q=1$ in Theorem \ref{MT2} and notice that $\left( p,q\right) \in
\Omega _{p,q}$, we get

\begin{corollary}
The double inequality 
\begin{equation}
\left( 1-\tfrac{p_{1}}{3}+\tfrac{p_{1}}{3}\cosh x\right) ^{1/p_{1}}<\frac{%
\sinh x}{x}<\left( 1-\tfrac{p_{2}}{3}+\tfrac{p_{2}}{3}\cosh x\right)
^{1/p_{2}}  \label{MI2a}
\end{equation}%
holds if and only if $p_{1}\geq 7/5$ and $0\leq p_{2}\leq 1$.
\end{corollary}

\begin{remark}
Letting $p_{1}=7/5,3/2,2,3$ and using the decreasing property of $M\left(
\cosh x;p,q\right) $ with respect to $p$, we can obtain the following chain
of inequalities from (\ref{MI2a}):%
\begin{eqnarray*}
\cosh ^{1/3}x &<&\left( \tfrac{1}{3}+\tfrac{2}{3}\cosh x\right)
^{1/2}<\left( \tfrac{1}{2}+\tfrac{1}{2}\cosh x\right) ^{2/3} \\
&<&\left( \tfrac{8}{15}+\tfrac{7}{15}\cosh x\right) ^{5/7}<\frac{\sinh x}{x}<%
\tfrac{2}{3}+\tfrac{1}{3}\cosh x.
\end{eqnarray*}%
Clearly, this chain of inequalities is superior to Che and S\'{a}ndor's
given in \cite[(3.23)]{Chen.JMI.8.1.2014}.
\end{remark}

Taking $p=0,1$ in Theorem \ref{MT3} and notice that $\left( p,q\right) \in
\Omega _{p,q}$, we get

\begin{corollary}
(i) The double inequality 
\begin{equation}
\exp \frac{\cosh ^{q_{1}}x-1}{3q_{1}}<\frac{\sinh x}{x}<\exp \frac{\cosh
^{q_{2}}x-1}{3q_{2}}  \label{MI3a}
\end{equation}%
holds if and only if $q_{1}\leq 0$ and $q_{2}\geq 8/15$.

(ii) The double inequality%
\begin{equation}
1-\tfrac{1}{3q_{1}}+\tfrac{1}{3q_{1}}\cosh ^{q_{1}}x<\frac{\sinh x}{x}<1-%
\tfrac{1}{3q_{2}}+\tfrac{1}{3q_{2}}\cosh ^{q_{2}}x  \label{MI3b}
\end{equation}%
holds for $q_{1}\leq 17/23\approx 0.73913$ and if and only if $q_{2}\geq 1$.
\end{corollary}

\begin{remark}
Letting $q_{1}=17/23,2/3,1/2,1/3,1/6,0$ and using the increasing property of 
$M\left( \cosh x;p,q\right) $ in $q$, we get the following chain of
inequalities from (\ref{MI3b}):%
\begin{eqnarray*}
\tfrac{1}{3}\ln \cosh x+1 &<&2\cosh ^{1/6}x-1<\cosh ^{1/3}x<\tfrac{1}{3}+%
\tfrac{2}{3}\cosh ^{1/2}x \\
&<&\tfrac{1}{2}+\tfrac{1}{2}\cosh ^{2/3}x<\tfrac{28}{51}+\tfrac{23}{51}\cosh
^{17/23}x<\frac{\sinh x}{x}<\tfrac{2}{3}+\tfrac{1}{3}\cosh x.
\end{eqnarray*}
\end{remark}

Let $p=kq$. Then $\Omega _{p,q}=\{\left( p,q\right) :p\geq 0$ or $3q\leq
p\leq 0\}$ is changed into 
\begin{equation}
\Omega _{kq,q}=\{\left( k,q\right) :k,q\geq 0\text{ or }k,q\leq 0\text{ or }%
k\in \lbrack 0,3],q\leq 0\},  \label{Op=kq}
\end{equation}%
while $M\left( t;p,q\right) $ can be expressed as 
\begin{equation}
M\left( t;kq,q\right) =\left\{ 
\begin{array}{ll}
\bigskip \left( 1-\frac{k}{3}+\frac{k}{3}t^{q}\right) ^{1/\left( kq\right) }
& \text{if }kq\neq 0,\left( k,q\right) \in \Omega _{kq,q}, \\ 
\bigskip \exp \frac{t^{q}-1}{3q} & \text{if }q\neq 0,k=0, \\ 
t^{1/3} & \text{if }q=0.%
\end{array}%
\right.  \label{M p=kq}
\end{equation}

\begin{remark}
Similar to the monotonicity of $M\left( t;p,q\right) $, we claim that $%
M\left( t;kq,q\right) $ is decreasing (increasing) in $q$ if $k>\left(
<\right) 3$, and is decreasing (increasing) in $k$ if $q>\left( <\right) 0$.

In fact, logarithmic differentiations gives%
\begin{eqnarray*}
\frac{\partial \ln M}{\partial q} &=&\frac{1}{q^{2}}\left( \frac{qt^{q}\ln t%
}{3-k+kt^{q}}-\frac{1}{k}\ln \left( 1-\frac{k}{3}+\frac{k}{3}t^{q}\right)
\right) :=\frac{M_{2}\left( t;k,q\right) }{q^{2}}, \\
\frac{\partial M_{2}}{\partial q} &=&\frac{t^{q}\ln ^{2}t}{\left(
3-k+kt^{q}\right) ^{2}}q\left( 3-k\right) ,
\end{eqnarray*}%
which means that $M_{2}$ is decreasing (increasing) in $q$ on $(0,\infty )$
and increasing (decreasing) on $\left( -\infty ,0\right) $ if $k>\left(
<\right) 3$. Hence we have $M_{2}\left( t;k,q\right) <\left( >\right)
M_{2}\left( t;k,0\right) =0$ if $k>\left( <\right) 3$, which reveals that $M$
is decreasing (increasing) in $q$ for $k>\left( <\right) 3$.

Analogously, the monotonicity of $M\left( t;kq,q\right) $ with respect to $k$
easily follows from the following relations:%
\begin{eqnarray*}
\frac{\partial \ln M}{\partial k} &=&\frac{1}{k^{2}}\left( \frac{k}{q}\frac{%
t^{q}-1}{3-k+kt^{q}}-\frac{1}{q}\ln \left( 1-\frac{k}{3}+\frac{k}{3}%
t^{q}\right) \right) :=\frac{M_{3}\left( t;k,q\right) }{k^{2}}, \\
\frac{\partial M_{3}}{\partial k} &=&-\frac{k}{q}\frac{\left( t^{q}-1\right)
^{2}}{\left( 3-k+kt^{q}\right) ^{2}}.
\end{eqnarray*}
\end{remark}

Using Corollary \ref{Corollary p=kq} we get

\begin{theorem}
\label{MT4}Let $x\in (0,\infty )$ and $k\in \lbrack 0,3)$. Then

(i) when $k\in \left[ 23/17,3\right) $, the inequality%
\begin{equation}
\frac{\sinh x}{x}>\left( 1-\tfrac{k}{3}+\tfrac{k}{3}\cosh ^{q}x\right)
^{1/\left( kq\right) }  \tag{MI4}  \label{MI4}
\end{equation}%
holds if and only if $q\leq 8/\left( 5\left( 3-k\right) \right) $;

(ii) when $k\in \lbrack 0,1]$, the double inequality%
\begin{equation}
\left( 1-\tfrac{k}{3}+\tfrac{k}{3}\cosh ^{q_{1}}x\right) ^{1/\left(
kq_{1}\right) }<\frac{\sinh x}{x}<\left( 1-\tfrac{k}{3}+\tfrac{k}{3}\cosh
^{q_{2}}x\right) ^{1/\left( kq_{2}\right) }  \tag{MI5}  \label{MI5}
\end{equation}%
holds if and only if $q_{1}\leq 0$ and $q_{2}\geq 8/\left( 5\left(
3-k\right) \right) $.
\end{theorem}

\begin{proof}
(i) In the case of $k\in \left[ 23/17,3\right) $. As shown previously, we
see that the inequality (\ref{MI4}) is equivalent to $D_{kq,q}\left(
x\right) =Sh_{kq}\left( x\right) -\left( 1/3\right) Ch_{q}\left( x\right) >0$%
. Then, by Corollary \ref{Corollary p=kq}, we see that (\ref{MI4}) holds for 
$0\leq q\leq 8/\left( 5\left( 3-k\right) \right) $. For $q<0$, since $%
M^{k}\left( t;kq,q\right) $ is a weighted power mean of order $q$ of
positive numbers $1$ and $\cosh x$, so we have $M^{k}\left( t;kq,q\right)
<M^{k}\left( t;k\times 0,0\right) $, and then (\ref{MI4}) still holds, which
proves the sufficiency. The necessity can be derived from $%
\lim_{x\rightarrow 0^{+}}x^{-4}D_{kq,q}\left( x\right) \geq 0$, which by \ref%
{Limit1} gives%
\begin{equation*}
\lim_{x\rightarrow 0^{+}}\frac{D_{kq,q}\left( x\right) }{x^{4}}=\frac{1}{72}%
\left( kq-3q+\frac{8}{5}\right) \geq 0.
\end{equation*}%
Solving the inequality for $q$ leads to $q\leq 8/\left( 5\left( 3-k\right)
\right) $.

(ii) In the case of $k\in (0,1]$. The sufficiency follows from Corollary \ref%
{Corollary p=kq}. It remains to prove the necessity. If the second
inequality in (\ref{MI5}) holds for $x\in \left( 0,\infty \right) $, then by %
\ref{Limit1} we have%
\begin{equation*}
\lim_{x\rightarrow 0^{+}}\frac{D_{kq_{2},q_{2}}\left( x\right) }{x^{4}}=%
\frac{1}{72}\left( kq_{2}-3q_{2}+\frac{8}{5}\right) \leq 0,
\end{equation*}%
which implies $q_{2}\geq 8/\left( 5\left( 3-k\right) \right) $. Lastly, we
show that the condition $q_{1}\leq 0$ is necessary for the first inequality
in (\ref{MI5}) to be true. If $q_{1}>0$, then $0<kq_{1}\leq q_{1}$. From (%
\ref{Limit2p,q>=0}) we know that $\lim_{x\rightarrow \infty
}e^{-q_{1}x}D_{kq_{1},q_{1}}\left( x\right) =-2^{-q_{1}}/\left(
3q_{1}\right) <0$, which means that there is an enough large number $x_{N}$
such that $D_{kq_{1},q_{1}}\left( x\right) <0$ for $x>x_{N}$, this
contradict with the fact that $D_{kq_{1},q_{1}}\left( x\right) >0$ for $x\in
\left( 0,\infty \right) $.

This theorem is proved.
\end{proof}

Taking $k=1,3/2,2$ in Theorem \ref{MT4}, we get

\begin{corollary}
\label{Corollary MT4p=kq}

(i) The double inequality%
\begin{equation}
\left( \tfrac{2}{3}+\tfrac{1}{3}\cosh ^{q_{1}}x\right) ^{1/q_{1}}<\frac{%
\sinh x}{x}<\left( \tfrac{2}{3}+\tfrac{1}{3}\cosh ^{q_{2}}x\right) ^{1/q_{2}}
\label{MI5a}
\end{equation}%
holds if and only if $q_{1}\leq 0$ and $q_{2}\geq 4/5$.

(ii) The inequality%
\begin{equation}
\frac{\sinh x}{x}>\left( \tfrac{1}{2}+\tfrac{1}{2}\cosh ^{q}x\right)
^{2/\left( 3q\right) }  \label{MI4a}
\end{equation}%
holds if and only if $q\leq 16/15$.

(iii) The inequality%
\begin{equation}
\frac{\sinh x}{x}>\left( \tfrac{1}{3}+\tfrac{2}{3}\cosh ^{q}x\right)
^{1/\left( 2q\right) }  \label{MI4b}
\end{equation}%
holds if and only if $q\leq 8/5$.
\end{corollary}

\begin{remark}
Part (i) in corollary above is exactly Theorem Zhu2.
\end{remark}

We close this section by considering the case of $p=3q-8/5$. In this case, $%
\Omega _{p,q}=\{\left( p,q\right) :p\geq 0$ or $3q\leq p\leq 0\}$ is changed
into 
\begin{equation}
\Omega _{3q-8/5,q}=\{3q-8/5\geq 0\text{ or }3q\leq 3q-8/5\leq 0\}=\{q\geq 
\frac{8}{15}\},  \label{Op=3q-8/5}
\end{equation}%
while $M\left( t;p,q\right) $ can be expressed as 
\begin{equation}
M\left( t;3q-\tfrac{8}{5},q\right) =\left\{ 
\begin{array}{ll}
\bigskip \left( \frac{8}{15q}+\left( 1-\frac{8}{15q}\right) t^{q}\right)
^{5/\left( 15q-8\right) } & \text{if }q>\tfrac{8}{15}, \\ 
\exp \frac{5\left( t^{8/15}-1\right) }{8} & \text{if }q=\tfrac{8}{15},%
\end{array}%
\right.  \label{M p=3q-8/5}
\end{equation}%
where $t=\cosh x\in \left( 1,\infty \right) $ for $x>0$. We assert that $%
M\left( t;3q-8/5,q\right) $ is decreasing in $q\in \lbrack 8/15,\infty )$.
Indeed, for $q>8/15$, logarithmic differentiation yields 
\begin{eqnarray*}
\frac{\partial \ln M}{\partial q} &=&-3\frac{\ln \left( \frac{8}{15q}+\left(
1-\frac{8}{15q}\right) t^{q}\right) }{\left( 3q-\frac{8}{5}\right) ^{2}}-%
\frac{\frac{8}{15q^{2}}\left( 1-t^{q}\right) +t^{q}\left( \frac{8}{15q}%
-1\right) \ln t}{\left( \frac{8}{15q}+\left( 1-\frac{8}{15q}\right)
t^{q}\right) \left( 3q-\frac{8}{5}\right) }, \\
\frac{\partial \ln M}{\partial t} &=&5q\frac{t^{q-1}}{\left( 15q-8\right)
t^{q}+8}, \\
\frac{\partial ^{2}\ln M}{\partial q\partial t} &=&40t^{q}\frac{\ln
t^{q}-t^{q}+1}{t\left( \left( 15q-8\right) t^{q}+8\right) ^{2}}<0,
\end{eqnarray*}%
where the inequality holds due to $\ln x\leq x-1$ for $x>0$. Hence, $%
\partial \left( \ln M\right) /\partial q$ is decreasing in $t$, and so we
have%
\begin{equation*}
\frac{\partial \ln M}{\partial q}\left( t;3q-8/5,q\right) <\frac{\partial
\ln M}{\partial q}\left( 1;3q-8/5,q\right) =0,
\end{equation*}%
which means that $q\mapsto M\left( t;3q-8/5,q\right) $ has decreasing
property. Now we show that%
\begin{equation}
\lim_{q\rightarrow \infty }M\left( t;3q-8/5,q\right) =t^{1/3}.
\label{Mlimit}
\end{equation}%
Employing L'Hospital rule yields%
\begin{eqnarray*}
&&\lim_{q\rightarrow \infty }\ln M\left( t;3q-8/5,q\right) \\
&=&5\lim_{q\rightarrow \infty }\frac{\ln \left( \left( 15q-8\right)
t^{q}+8\right) -\ln \left( 15q\right) }{15q-8}=\frac{1}{3}\lim_{q\rightarrow
\infty }\left( \frac{t^{q}\left( 15q\ln t-8\ln t+15\right) }{%
15qt^{q}-8t^{q}+8}-\frac{1}{q}\right) \\
&=&\frac{1}{3}\lim_{q\rightarrow \infty }\left( \frac{15\ln t-8q^{-1}\ln
t+15q^{-1}}{15-8q^{-1}+8q^{-1}t^{-q}}-\frac{1}{q}\right) =\frac{1}{3}\ln t,
\end{eqnarray*}%
that is, (\ref{Mlimit}) is valid.

\begin{theorem}
\label{MT5}Let $x\in (0,\infty )$ and $q>8/15$. Then the inequality%
\begin{equation}
\frac{\sinh x}{x}>\left( \frac{8}{15q}+\left( 1-\frac{8}{15q}\right) \cosh
^{q}x\right) ^{5/\left( 15q-8\right) }  \tag{MI6}  \label{MI6}
\end{equation}%
holds true if and only if $q\geq 34/35$. Its reverse holds if and only if $%
q\leq 4/5$.
\end{theorem}

\begin{proof}
The sufficiency is obviously a consequence of Corollary \ref{Corollary
p=3q-8/5}. The necessity such that (\ref{MI6}) holds due to $%
\lim_{x\rightarrow 0^{+}}x^{-6}D_{3q-8/5,q}\left( x\right) \geq 0$, which
together with (\ref{Limit1 p=3q-8/5}) yields $q\geq 34/35$. It remains to
treat the necessity such that the reverse of (\ref{MI6}). Due to the
decreasing property of $M\left( \cosh x;3q-8/5,q\right) $, if there is a
more large number $q^{\ast }>4/5$ such that reverse of (\ref{MI6}) holds,
which is equivalent to $D_{3q^{\ast }-8/5,q^{\ast }}\left( x\right) <0$ for $%
x\in \left( 0,\infty \right) $, then $3q^{\ast }-8/5>q^{\ast }>4/5$. From (%
\ref{Limit2p,q>=0}) we get $\lim_{x\rightarrow \infty }e^{-q^{\ast
}x}D_{3q^{\ast }-8/5,q^{\ast }}\left( x\right) =\infty $, which implies that
there is an enough large number $x_{N}$ such that $D_{3q^{\ast }-8/5,q^{\ast
}}\left( x\right) >0$ for $x>x_{N}$. This contradict with the fact that $%
D_{3q^{\ast }-8/5,q^{\ast }}\left( x\right) <0$ for $x\in \left( 0,\infty
\right) $, therefore, the constant $4/5$ is the best.

Thus the proof of this theorem is complete.
\end{proof}

Putting $q=34/35,1,16/15,6/5,8/5,2,\infty $ and $4/5,7/10,2/3,3/5,8/15^{+}$
in Theorem \ref{MT5} we have

\begin{corollary}
For $x\in (0,\infty )$, the chain of inequalities hold:%
\begin{eqnarray*}
\cosh ^{1/3}x &<&\cdot \cdot \cdot <\left( \tfrac{11}{15}\cosh ^{2}x+\tfrac{4%
}{15}\right) ^{5/22}<\left( \tfrac{2}{3}\cosh ^{8/5}x+\frac{1}{3}\right)
^{5/16}< \\
\left( \tfrac{5}{9}\cosh ^{6/5}x+\tfrac{4}{9}\right) ^{1/2} &<&\left( \tfrac{%
1}{2}\cosh ^{16/15}x+\tfrac{1}{2}\right) ^{5/8}<\left( \tfrac{7}{15}\cosh x+%
\tfrac{8}{15}\right) ^{5/7}< \\
\left( \tfrac{23}{51}\cosh ^{34/35}x+\tfrac{28}{51}\right) ^{35/46} &<&\frac{%
\sinh x}{x}<\left( \tfrac{1}{3}\cosh ^{4/5}x+\tfrac{2}{3}\right)
^{5/4}<\left( \tfrac{5}{21}\cosh ^{7/10}x+\tfrac{16}{21}\right) ^{2}< \\
\left( \tfrac{1}{5}\cosh ^{2/3}x+\tfrac{4}{5}\right) ^{5/2} &<&\cdot \cdot
\cdot <\exp \left( \tfrac{5}{8}\cosh ^{8/15}x-\tfrac{5}{8}\right) .
\end{eqnarray*}
\end{corollary}

\section{Inequalities for means}

Let $G,A,Q$ and $L$ stand for the geometric, arithmetic, quadratic and
logarithmic means of any positive numbers $a$ and $b$ defined by%
\begin{eqnarray*}
G &=&G\left( a,b\right) =\sqrt{ab}\text{, \ }A=A\left( a,b\right) =\frac{a+b%
}{2}\text{, \ }Q=Q\left( a,b\right) =\sqrt{\frac{a^{2}+b^{2}}{2}}, \\
L &=&L\left( a,b\right) =\frac{a-b}{\ln a-\ln b}\text{ if }a\neq b\text{ \
and \ }L=L\left( a,a\right) =a.
\end{eqnarray*}%
The Schwab-Borchardt mean of two numbers $a\geq 0$ and $b>0$, denoted by $%
SB(a,b)$, is defined as \cite[Theorem 8.4]{Biernacki.9.1955}, \cite[3, (2.3)]%
{Carlson.78(1971)}%
\begin{equation*}
SB(a,b)=\left\{ 
\begin{array}{cc}
\frac{\sqrt{b^{2}-a^{2}}}{\arccos (a/b)} & \text{if\ }a<b, \\ 
a & \text{if \ }a=b, \\ 
\frac{\sqrt{a^{2}-b^{2}}}{\func{arccosh}(a/b)} & \text{if \ }a>b.%
\end{array}%
\right.
\end{equation*}%
The properties and certain inequalities involving Schwab-Borchardt mean can
be found in \cite{Neuman.14(2003)}, \cite{Neuman.JMI.2012.inprint}. Very
recently, Yang \cite[Theorem 7.1]{Yang.JIA.2013.541} has defined a family of
two-parameter hyperbolic sine means as follows.

\begin{definition}
Let $p,q\in \mathbb{R}$ and let $Sh(p,q,t)$ be defined by%
\begin{equation}
Sh\left( p,q,t\right) =\left\{ 
\begin{array}{ll}
\left( \dfrac{q}{p}\dfrac{\sinh pt}{\sinh qt}\right) ^{1/\left( p-q\right) }
& \text{if }pq\left( p-q\right) \neq 0, \\ 
\left( \dfrac{\sinh pt}{pt}\right) ^{1/p} & \text{if }p\neq 0,q=0, \\ 
\left( \dfrac{\sinh qt}{qt}\right) ^{1/q} & \text{if }p=0,q\neq 0, \\ 
e^{t\coth pt-1/p} & \text{if }p=q,pq\neq 0, \\ 
1 & \text{if }p=q=0.%
\end{array}%
\right.  \label{Sh(p,q,t)}
\end{equation}%
Then for all $b\geq a>0$, $Sh_{p,q}\left( b,a\right) $ defined by%
\begin{equation}
Sh_{p,q}\left( b,a\right) =a\times Sh\left( p,q,\func{arccosh}\left(
b/a\right) \right) \text{ if }a<b\text{ \ and \ }Sh_{p,q}\left( a,a\right) =a
\label{TShmean}
\end{equation}%
is a mean of $a$ and $b$ if $\left( p,q\right) $ satisfies%
\begin{equation*}
\begin{array}{ll}
p+q\leq 3\text{ \ and \ }L\left( p,q\right) \leq \frac{1}{\ln 2}, & \text{if 
}p,q>0, \\ 
0\leq p+q\leq 3, & \text{otherwise,}%
\end{array}%
\end{equation*}%
where $L\left( p,q\right) $ is the logarithmic mean of positive numbers $p$
and $q$.
\end{definition}

As as a special case, for $b\geq a>0$, 
\begin{equation*}
Sh_{1,0}\left( b,a\right) =a\frac{\sinh t}{t}\Big |_{t=\func{arccosh}\left(
b/a\right) }
\end{equation*}%
is a mean of $a$ and $b$. Clearly, $Sh_{1,0}\left( b,a\right) =SB\left(
b,a\right) $. Thus, after replacing $t$ by $\func{arccosh}\left( b/a\right) $
and multiplying each sides of those inequalities showed in previous section
by $a$, Theorems \ref{MT2}--\ref{MT5} still hold, for example, Theorems \ref%
{MT2}--\ref{MT5} can be restated as follows.

\begin{theorem}
\label{MT2a}Let $b\geq a>0$ and $\left( p,q\right) \in \Omega
_{p,q}=\{\left( p,q\right) :p\geq 0$ or $3q\leq p\leq 0\}$. Then

(i) when $q\in \left[ 34/35,\infty \right) $, the double inequality%
\begin{equation}
\left( \left( 1-\tfrac{p_{1}}{3q}\right) a^{q}+\tfrac{p_{1}}{3q}b^{q}\right)
^{1/p_{1}}a^{1-q/p_{1}}<SB\left( b,a\right) <\left( \left( 1-\tfrac{p_{2}}{3q%
}\right) a^{q}+\tfrac{p_{2}}{3q}b^{q}\right) ^{1/p_{2}}a^{1-q/p_{2}} 
\tag{MI2`}  \label{MI2`}
\end{equation}%
holds if and only if $p_{1}\geq 3q-8/5$ and $p_{2}\leq q$;

(ii) when $q\in \left[ 4/5,34/35\right) $, the double inequality (\ref{MI2`}%
) holds for $p_{1}\geq 23q/17$ and if and only if $p_{2}\leq q$;

(iii) when $q\in (0,4/5)$, the double inequality (\ref{MI2`}) holds for $%
p_{1}\geq 23q/17$ and if and only if $p_{2}\leq 3q-8/5$;

(iv) when $q\in (-\infty ,0]$, the double inequality (\ref{MI2`}) holds if
and only if $p_{1}\geq q$ and $p_{2}\leq 3q-8/5$.
\end{theorem}

\begin{theorem}
\label{MT3a}Let $b\geq a>0$ and $\left( p,q\right) \in \Omega
_{p,q}=\{\left( p,q\right) :p\geq 0$ or $3q\leq p\leq 0\}$. Then

(i) when $p\in \left[ 46/35,\infty \right) $, the double inequality%
\begin{equation}
\left( \left( 1-\tfrac{p}{3q_{1}}\right) a^{q_{1}}+\tfrac{p}{3q_{1}}%
b^{q_{1}}\right) ^{1/p}a^{1-q_{1}/p}<SB\left( b,a\right) <\left( \left( 1-%
\tfrac{p}{3q_{2}}\right) a^{q_{2}}+\tfrac{p}{3q_{2}}b^{q_{2}}\right)
^{1/p}a^{1-q_{2}/p}  \tag{MI3`}  \label{MI3`}
\end{equation}%
holds if and only if $q_{1}\leq p/3+8/15$ and $q_{2}\geq p$;

(ii) when $p\in \left[ 4/5,46/35\right) $, the double inequality (\ref{MI3`}%
) holds for $q_{1}\leq 17p/23$ and if and only if $q_{2}\geq p$;

(iii) when $p\in (0,4/5)$, the double inequality (\ref{MI3`}) holds for $%
q_{1}\leq 17p/23$ and if and only if $q_{2}\geq p/3+8/15$;

(iv) when $p\in (-\infty ,0]$, the double inequality (\ref{MI3`}) holds if
and only if $q_{1}\leq p$ and $q_{2}\geq p/3+8/15$.
\end{theorem}

\begin{theorem}
\label{MT4a}Let $b\geq a>0$ and $k\in \lbrack 0,3)$. Then

(i) when $k\in \left[ 23/17,3\right) $, the inequality%
\begin{equation}
SB\left( b,a\right) >\left( \left( 1-\tfrac{k}{3}\right) a^{q}+\tfrac{k}{3}%
b^{q}\right) ^{1/\left( kq\right) }a^{1-1/k}  \tag{MI4`}
\end{equation}%
holds if and only if $q\leq 8/\left( 5\left( 3-k\right) \right) $;

(ii) when $k\in \lbrack 0,1]$, the double inequality%
\begin{equation}
\left( \left( 1-\tfrac{k}{3}\right) a^{q_{1}}+\tfrac{k}{3}b^{q_{1}}\right)
^{1/\left( kq_{1}\right) }a^{1-1/k}<SB\left( b,a\right) <\left( \left( 1-%
\tfrac{k}{3}\right) a^{q_{2}}+\tfrac{k}{3}b^{q_{2}}\right) ^{1/\left(
kq_{2}\right) }a^{1-1/k}  \tag{MI5`}
\end{equation}%
holds if and only if $q_{1}\leq 0$ and $q_{2}\geq 8/\left( 5\left(
3-k\right) \right) $.
\end{theorem}

\begin{theorem}
\label{MT5a}Let $b\geq a>0$ and $q>8/15$. Then the inequality%
\begin{equation}
SB\left( b,a\right) >\left( \frac{8}{15q}a^{q}+\left( 1-\frac{8}{15q}\right)
b^{q}\right) ^{5/\left( 15q-8\right) }a^{-q/\left( 15q-8\right) }  \tag{MI6'}
\end{equation}%
holds true if and only if $q\geq 34/35$. Its reverse holds if and only if $%
q\leq 4/5$.
\end{theorem}

Further, let $m=m(a,b)$ and $M=M(a,b)$ be two means of $a$ and $b$ with $%
m(a,b)<M(a,b)$ for all $a,b>0$. Clearly, making a change of variables $%
a\rightarrow m\left( a,b\right) $ and $b\rightarrow M(a,b)$, $Sh_{p,q}\left(
M,m\right) $ is still a mean of $a$ and $b$ which lie in $m$ and $M$.
Particularly, taking $\left( M,m\right) =\left( A,G\right) $, $\left(
Q,A\right) $, $\left( Q,G\right) $, respectively, we can obtain new
symmetric means as follows:%
\begin{eqnarray*}
Sh_{1,0}\left( A,G\right)  &=&SB\left( A,G\right) =\frac{a-b}{\ln a-\ln b}%
=L\left( a,b\right) , \\
Sh_{1,0}\left( Q,A\right)  &=&SB\left( Q,A\right) =\frac{a-b}{2\func{arcsinh}%
\frac{a-b}{a+b}}=NS\left( a,b\right) , \\
Sh_{1,0}\left( Q,G\right)  &=&SB\left( Q,G\right) =\frac{a-b}{\sqrt{2}\func{%
arcsinh}\frac{a-b}{\sqrt{2ab}}}=V\left( a,b\right) ,
\end{eqnarray*}%
where $NS\left( a,b\right) $ is Neuman-S\'{a}ndor mean first given by \cite%
{Neuman.14(2003)}, $V\left( a,b\right) $ is a new mean first appeared in 
\cite{Yang.JIA.2013.541}.

Thus, after replacing $\left( b,a,SB\left( b,a\right) \right) $ with $\left(
A,G,L\right) $, $\left( Q,A,NS\right) $, $\left( Q,G,V\right) $, Theorems %
\ref{MT2a}--\ref{MT4a} are still true.

\end{document}